%% file: main.tex
\documentclass[a4paper,reqno,11pt]{amsart}

\usepackage{amsmath}
\usepackage{amssymb}
\usepackage{cases}
\usepackage{enumitem}
\allowdisplaybreaks[4]
\usepackage{mathrsfs}
\usepackage{yfonts}
\input xy
\xyoption{all}
\usepackage{mathtools}
\DeclarePairedDelimiter{\ceil}{\lceil}{\rceil}
\DeclarePairedDelimiter{\floor}{\lfloor}{\rfloor}

\newtheorem{thm}{Theorem}[section]
\newtheorem{prop}[thm]{Proposition}

\newtheorem{lem}[thm]{Lemma}

\newtheorem{cor}[thm]{Corollary}

\theoremstyle{definition}
\newtheorem{definition}[thm]{Definition}

\theoremstyle{remark}
\newtheorem{remark}[thm]{Remark}

\numberwithin{equation}{section}

\newcommand{\bQ}{\mathbb{Q}}
\newcommand{\bP}{\mathbb{P}}

\newcommand\OO{{\mathcal{O}}}

\newcommand{\bR}{{\mathbb R}}

\newcommand\lct{{\rm{lct}}}

\begin{document}

\title{Boundedness of varieties of Fano type with alpha-invariants and volumes bounded below}
\date{\today}

\author{Weichung Chen}
\address{Graduate School of Mathematical Sciences, the University of Tokyo, Tokyo 153-8914, Japan.}
\email{chenweic@ms.u-tokyo.ac.jp}

\begin{abstract}
We show that fixed dimensional klt weak Fano pairs with alpha-invariants and volumes bounded away from $0$ and the coefficients of the boundaries belong to the set of hyperstandard multiplicities $\Phi(\mathscr{R})$ associated to a fixed finite set $\mathscr{R}$ form a bounded family. We also show $\alpha(X,B)^{d-1}\text{vol}(-(K_X+B))$ are bounded from above for all klt weak Fano pairs $(X,B)$ of a fixed dimension $d$.
\end{abstract}

\maketitle
\pagestyle{myheadings} \markboth{\hfill  Weichung Chen
\hfill}{\hfill  \hfill}


\input{1.intro}
\input{2.preliminaries}
\input{3.pf_of_main_main2}

\appendix

\input{biblio}
\end{document}

%% file: 1.intro.tex
\section{Introduction}
Throughout this paper, we work over an uncountable algebraically closed field
of characteristic 0, for instance, the complex number field $\mathbb{C}$.

In the birational geometry, as the first step of moduli theory, it is interesting to consider whether a certain kind of family of varieties satisfy certain finiteness. For varieties of Fano type with bounded log discrepancies, Birkar shows in {\cite [Theorem 1.1]{Bir16b}} that
\begin{thm}\label{bab}
Fix a positive integer $d$ and a positive real number $\epsilon$. The projective varieties $X$ satisfying
\begin{enumerate}
\item $\dim X=d,$
\item there exists a boundary $B$ such that $(X, B)$ is $\epsilon$-lc,
\item $-(K_X+B)$ is nef and big,
\end{enumerate}
form a bounded family.
\end{thm}
Theorem \ref{bab} was known as the Borisov-Alexeev-Borisov (BAB) Conjecture for decades before Birkar proved it. Equivalently, we can state Theorem \ref{bab} in the following form of boundedness of varieties of Calabi--Yau type.

\begin{thm}\label{bab2}
Fix a positive integer $d$ and a positive real number $\epsilon$. The projective varieties $X$ satisfying
\begin{enumerate}
\item $\dim X=d,$
\item there exists a boundary $B$ such that $(X, B)$ is $\epsilon$-lc,
\item $K_X+B \sim_{\mathbb{R}}0$ and $B$ is big,
\end{enumerate}
form a bounded family.
\end{thm}

 In Theorem \ref{bab}, it is necessary to take $\epsilon>0$. In fact, klt Fano threefolds do not even form a birational family (see {\cite {Lin03}}). Nevertheless, Jiang shows in {\cite {Jia17}} that if we bound the alpha-invariants and the volumes from the below, we have 

\begin{thm}$\mathrm{(}${\cite[Theorem1.6]{Jia17}}\label{jiangbdd}$\mathrm{)}$
Fix a positive integer $d$ and a positive real number $\theta$. The normal projective  klt Fano (i.e. $\bQ$-Fano in \cite {Jia17}) varieties $X$ satisfying
\begin{enumerate}
\item $\dim X=d$,
\item $\alpha (X)^d(-K_X)^d>\theta$,
\end{enumerate}
form a bounded family.
\end{thm}

Inspired by Theorem \ref{jiangbdd}, it is natural to ask if certain boundedness holds for varieties of Fano type or under other more general setting. Thanks to boundedness of complements by Birkar (Theorem \ref{bddcomp}), if the coefficients of the boundaries are well controlled, then the boundedness, which is one of our main theorems, holds as follows.

\begin{thm}\label{main}
Fix a positive integer $d$, positive real numbers $\theta$ and $\delta$ and a finite set $\mathscr{R}$ of rational numbers in $[0,1]$. The set of all klt Fano pairs $(X,B)$ satisfying
\begin{enumerate}
\item $\dim X=d$,
\item the coefficients of $B\in\Phi(\mathscr{R})$,
\item $\alpha(X,B)^{d-1+\delta}(-(K_X+B))^d>\theta$,
\end{enumerate}
forms a log bounded family.
\end{thm}

Letting $B=0$ in Theorem \ref{main}, we have the following corollary, which answers the question asked by Jiang in {\cite[1.7]{Jia17}}.
\begin{cor}
Fix a positive integer $d$ and two positive real numbers $\delta$ and $\theta$. Then the set of klt Fano varieties $X$ satisfying
\begin{enumerate}
\item $\dim X=d$,
\item $\alpha (X)^{d-1+\delta}(-K_X)^d>\theta$,
\end{enumerate}
forms a bounded family.
\end{cor}

Now we consider $\alpha(X,B)^{d}(-(K_X+B))^d$ as an invariant for $d$-dimensional klt Fano pairs $(X,B)$. It is well known that this invariant has an upper bound, which can be given by the following lemma.
\begin{lem}\label{kollar}$\mathrm{(}${\cite [Theorem 6.7.1]{Kol97}}$\mathrm{)}$
Let $(X,B)$ be a klt pair of dimension $d$. Then we have
$$\mathrm{lct}((X,B),|H|_{\bQ})^dH^d\leq d^d$$
for any nef and big $\bQ$-Cartier divisor $H$ on $X$.
\end{lem}
We will show that for $d$-dimensional klt Fano pairs $(X,B)$,
\[\alpha(X,B)^{d-1}(-(K_X+B))^d\]
is also bounded above. In fact, we have the following theorem under a more general setting.
\begin{thm}\label{thm1}
Fix a positive integer $d$. There exists a number $M$, depending only on $d$, such that for any projective normal pair $(X,B)$ and for any big and nef $\bQ$-Cartier divisor $H$ on $X$ satisfying
\begin{enumerate}
\item $\dim X=d$,
\item $(X, B)$ is klt,
\end{enumerate}
we have 
\[\mathrm{lct}((X,B),|H|_{\bQ})^{d-1}\overline{\tau}((X,B),H)H^d\leq M,\]
where $\overline{\tau}$ denotes the anti-pseudo-effective threshold (see Definition \ref{aet}).
\end{thm}

\begin{remark}
In contrast with the above, for $d$-dimensional klt Fano pairs $(X,B)$, $\alpha(X,B)^{d'}(-(K_X+B))^d$ are not bounded above if $d'<d-1$. To see this, consider the weighted projective spaces $X_n=\bP(1^d,n)$, which are klt Fano varieties of dimension $d$ with $(-K_{X_n})^d=\frac{(n+d)^d}{n}$ and $\alpha(K_{X_n})=\frac{1}{n+d}$ (cf. {\cite[6.3]{Amb16}}). Then we have, $\alpha(X)^{d-1-\delta}(-K_{X_n})^d=\frac{(n+d)^{(1+\delta)}}{n}$, which are not bounded above for any positive real number $\delta$.
\end{remark}
\begin{remark}
We remark that for a klt Fano variety $X$ of dimension $d$, a lower bound of $\alpha(X)$ provides an upper bound of $(-K_X)^d$ by Lemma \ref{kollar}. However, the set of klt Fano varieties with $(-K_X)^d$ both side bounded does not form a bounded family. As an example, consider the family of weighted projective spaces $\{X_{p,q,r}=\bP(p,q,r)\}$ with $(p,q,r)$ pairwisely coprime. Then $\{(-K_{X_{p,q,r}})^2\}=\{\frac{(p+q+r)^2}{pqr}\}$ is a dense subset of $\bR_{>0}$. Therefore, for any two positive integers $a<b$, $\{X_{p,q,r}|\frac{(p+q+r)^2}{pqr}\in(a,b)\}$ is a family of klt Fano varieties which is not bounded.
\end{remark}

\medskip

\noindent\textbf{Acknowledgments}. This paper is inspired by {\cite {Jia17}}. The author would like to thank Chen Jiang and his advisor Yoshinori Gongyo for many inspiring and useful discussions and suggestions. The author was supported by Japan--Taiwan Exchange Association Scholarship.

%% file: 2.preliminaries.tex
\section{Preliminaries}
We adopt the standard notation and definitions in \cite{KMM} 
and \cite{KM}, and will freely use them.

\subsection{Pairs and singularities}
A {\it sub-pair} $(X, B)$ consists of a normal projective variety $X$ 
and an $\bR$-divisor $B$ on $X$ such that $K_X+B$ is $\bR$-Cartier. $B$ is called the {\it sub-boundary} of this pair.

A {\it log pair} $(X, B)$ is a sub-pair with $B\geq 0$. We call $B$ a {\it boundary} in this case.

Let $f\colon Y\rightarrow X$ be a log
resolution of the log pair $(X, B)$, write
\[
K_Y =f^*(K_X+B)+\sum a_iF_i,
\]
where $\{F_i\}$ are distinct prime divisors.  
For a non-negative real number $\epsilon$, the log pair $(X,B)$ is called
\begin{itemize}
\item[(a)] \emph{$\epsilon$-kawamata log terminal} (\emph{$\epsilon$-klt},
for short) if $a_i> -1+\epsilon$ for all $i$;
\item[(b)] \emph{$\epsilon$-log canonical} (\emph{$\epsilon$-lc}, for
short) if $a_i\geq  -1+\epsilon$ for all $i$;
\end{itemize}

Usually we write $X$ instead of $(X,0)$ in the case when $B=0$.
Note that $0$-klt (resp. $0$-lc) is just klt (resp. lc) in the usual sense. Also note that 
$\epsilon$-lc singularities only make sense if $\epsilon\in [0,1]$, and  $\epsilon$-klt 
singularities only make sense if $\epsilon\in [0,1)$.

Similarly, sub-$\epsilon$-klt and sub-$\epsilon$-lc sub-pairs can be defined.

The {\it log discrepancy} of the divisor $F_i$ is defined to be $a(F_i, X, B)=1+a_i$.
It does not depend on the choice of the log resolution $f$.

$F_i$ is called a {\it non-lc place} of $(X, B)$  if $a_i< -1$.
A subvariety $V\subset X$ is called a {\it non-lc center} of 
$(X, B)$ if it is the image of a non-lc place. 
The {\it non-klt locus} $\text{Nklt}(X, B)$ is the union of 
all non-lc centers of $(X, B)$.
We recall the Koll\'{a}r-Shokurov connectedness lemma.
\begin{lem}\label{cnnlem}$\mathrm($cf. \cite{Sho93}, \cite{Sho94} and {\cite[17.4]{Kol92}}$\mathrm)$
Let $(X,B)$ be a log pair, and let $\pi$: $X\rightarrow S$ be a proper morphism with connected fibers. Suppose $-(K_X+B)$ is $\pi$-nef and $\pi$-big. Then $\mathrm{Nklt}(X, B)\cap X_s$ is connected for any fiber $X_s$ of $\pi$.
\end{lem}

\subsection{Fano pairs and Calabi--Yau pairs}
A normal projective pair $(X,B)$ is a {\it Fano} (resp. {\it weak Fano}, resp. {\it Calabi--Yau}) pair if $-(K_X+B)$ is ample (resp. $-(K_X+B)$ is nef and big, resp. $K_X+B\equiv 0$).
A normal projective variety $X$ is called Fano if $(X,0)$ is Fano. It is called {\it $\mathbb{Q}$-Fano} if it is klt and Fano. It is called {\it of Fano type} if $(X,B)$ is klt weak Fano for some boundary $B$.

\subsection{Bounded pairs}\label{sec.bdd}
A collection of varieties $ \mathcal{D}$ is
said to be \emph{bounded} (resp. 
\emph{birationally bounded})
if there exists 
$h\colon \mathcal{Z}\rightarrow S$ a projective morphism 
of schemes of finite type such that
each $X\in \mathcal{D}$ is isomorphic (resp. birational, 
or isomorphic in codimension one) to $\mathcal{Z}_s$ 
for some closed point $s\in S$.

A couple $(X,D)$ consists of a normal projective variety $X$ and a reduced divisor $D$ on X. Note that we do not require $K_X+D$ to be $\bQ$-Cartier here.

We say that a collection of couples $\mathcal{D}$ is 
{\it log birationally bounded} (resp.  \emph{log bounded})
if there is a  quasi-projective scheme $\mathcal{Z}$, a 
reduced divisor $\mathcal{E}$ on $\mathcal Z$, and a 
projective morphism $h\colon \mathcal{Z}\to S$, where 
$S$ is of finite type and $\mathcal{E}$ does not contain 
any fiber, such that for every $(X,D)\in \mathcal{D}$, 
there is a closed point $s \in S$ and a birational
map $f \colon \mathcal{Z}_s \dashrightarrow X$ 
(resp. isomorphic)
such that $\mathcal{E}_s$ contains the support of $f_*^{-1}B$ 
and any $f$-exceptional divisor.

A set of log pairs $\mathcal{P}$ is 
{\it log birationally bounded} (resp.  \emph{log bounded})
if the set of the corresponding couples $\{(X,\mathrm{Supp}B)|(X,B)\in\mathcal{P}\}$ is.

\subsection{Volumes}
Let $X$ be a $d$-dimensional projective variety  and $D$ 
a Cartier divisor on $X$. The {\it volume} of $D$ is the real number
\[
\mathrm{vol}(X, D)=\limsup_{m\rightarrow \infty}\frac{h^0(X,\OO_X(mD))}{m^d/d!}.
\]
For more backgrounds on the volume, see \cite[2.2.C]{Positivity1}. 
By the homogenous property and  continuity of the volume, we 
can extend the definition to $\bR$-Cartier $\bR$-divisors. 
Moreover, if $D$ is a nef $\bR$-divisor, then vol$(X, D)=D^d$.

\subsection{Complements}
\begin{definition}
Let $(X,B)$ be a pair and $n$ a positive integer. We write $B=\floor{B}+\{B\}$. An {\it n-complement} of $K_X+B$ is a divisor of the form $K_X+B^+$ such that 
\begin{enumerate}
\item $(X,B^+)$ is lc,
\item $n(X,B^+)\sim 0$,
\item $nB^+\leq n\floor{B}+\floor{(n+1)\{B\}}$.
\end{enumerate}

\end{definition}
\begin{definition}
For a subset $\mathscr{R}$ of $[0,1]$, we define the set of {\it hyperstandard multiplicities} associated to $\mathscr{R}$ to be
\[\Phi(\mathscr{R})=\{ 1-\frac{r}{m}| r\in\mathscr{R},\ m\in\mathbb{N} \}.\]
\end{definition}
Note that the only possible accumulating point of $\Phi(\mathscr{R})$  is $1$ if $\mathscr{R}$ is finite.
Birkar shows the following boundedness of complements.
\begin{thm}$\mathrm(${\cite [Theorem 1.7]{Bir16a}}$\mathrm)$\label{bddcomp}
Fix a positive integer $d$ and a finite set $\mathscr{R}$ of rational numbers in $[0,1]$. Then there exists a positive integer $n$ depending only on $d$ and $\mathscr{R}$, such that if  $(X,B)$ is a projective pair with
\begin{enumerate}
\item $(X,B)$ is lc dimension $d$,
\item the coefficients of $B\in\Phi(\mathscr{R})$,
\item $X$ is of Fano type,
\item $-(K_X+B)$ is nef,
\end{enumerate}
then there is an $n$-complement $K_X+B^+$ of $K_X+B$ such  that $B^+\geq B$.
\end{thm}

\subsection{$\alpha$-invariants, log canonical thresholds and anti-pseudo-effective thresholds}
\begin{definition}
Let $(X,B)$ be a projective lc pair and let $D$ be an effective ${\bR}$-Cartier divisor, we define the {\it log canonical threshold} of $D$ with respect of $(X,B)$ to be
\[\lct((X,B),D)=\sup\{\text{$t\in \mathbb{R}$ $|$ $(X,B+tD)$ is lc}\}.\]
The log canonical threshold of $|D|$ with respect of $(X,B)$ is defined to be
\[\lct((X,B),|D|_{\bQ})=\inf\{ \mathrm{lct}((X,B),M)| M\in|D|_{\bQ}\},\]
which is equal to
\[\lct((X,B),|D|_{\bR})=\inf\{ \mathrm{lct}((X,B),M)| M\in|D|_{\bR}\}.\]
\end{definition}
\begin{definition}
Let $(X,B)$ be a projective normal klt weak Fano pair, we define the {\it $\alpha$-invariant} of $(X,B)$ to be
\[\alpha(X,B)=\mathrm{lct}((X,B),|-(K_X+B)|_{\bQ}).\]
In the case when $B=0$, we usually write $\alpha(X):=\alpha(X,0)$ for convenience.
\end{definition}

\begin{definition}\label{aet}
Let $(X,B)$ be a projective normal pair, and $H$ a big $\mathbb{R}$-Cartier divisor. The {\it anti-pseudo-effective threshold} of $H$ respect to $(X,B)$ is defined by
 \begin{align*}
\text{$\overline{\tau}((X,B),H)$}&\text{$=$sup$\{ t\in\mathbb{R}|$ $-K_X-B-tH$ is pseudo-effective$\}$}\\
&\text{$=$sup$\{ t\in\mathbb{R}|$ $K_X+B+tH$ is anti-pseudo-effective$\}$.}
 \end{align*}
\end{definition}
\subsection{Potentially birational divisors}
\begin{definition}(cf. {\cite[Difinition 3.5.3]{HMX14}})
Let $X$ be a projective normal variety, and $D$ a big $\mathbb{Q}$-Cartier $\bQ$-divisor on $X$. Then, we say that $D$ is {\it potentially birational} if for any two general points $x$ and $y$ of $X$, there is an effective $\bQ$-divisor $\Delta\sim_{\bQ}(1-\epsilon)D$ for some $0<\epsilon<1$, such that, after possibly switching $x$ and $y$, $(X,\Delta)$ is not lc at $y$, lc at $x$ and $x$ is a non-klt center.
\end{definition}

\begin{lem}$\mathrm(${\cite[Lemma 2.3.4]{HMX13}}$\mathrm)$
Let $X$ be a projective normal variaty, and $D$ a big $\mathbb{Q}$-Cartier $\bQ$-divisor on $X$. If $D$ is potentially birational, then $|K_X+\ceil{D}|$ defines a birational map.
\end{lem}

\subsection{Exceptional pairs}
\begin{definition}
A projective normal klt pair $(X,B)$ is called {\it exceptional} if \[\mathrm{lct}((X,B),|-(K_X+B)|_{\bQ})>1.\]
In particular, if $(X,B)$ is weak Fano, then it is exceptional if and only if $\alpha(X,B)>1$.
\end{definition}
\begin{thm}\label{excbdd}$\mathrm(${\cite[Theorem 1.11]{Bir16a}}$\mathrm)$
Fix a positive integer $d$ and a finite set of rational numbers $\mathscr{R}$ in $[0,1]$. Then the set of all projective pairs $(X,B)$ satisfying
\begin{enumerate}
\item $(X,B)$ is lc dimension $d$,
\item the coefficients of $B\in\Phi(\mathscr{R})$,
\item $X$ is of Fano type,
\item $-(K_X+B)$ is nef,
\item $(X,B)$ is exceptional,
\end{enumerate}
forms a log bounded family.
\end{thm}
\subsection{Descending chain condition}
\begin{definition}
A set of real numbers $\mathscr{S}$ is said to {\it satisfy descending chain condition (DCC for short)} if for every non-empty subset $S$ of $\mathscr{S}$, there is a minimum element in $S$. $\mathscr{S}$ is called a DCC set if it satisfies DCC.
\end{definition}

We recall and improve slightly the following proposition of Birkar. It is shown for ample divisors $D$ and $A$ in {\cite[2.31(2)]{Bir16a}}. We modify it for nef and big divisors $A$ and $D$.
\begin{prop}\label{birkar}
Let $(X,B)$ be a log pair of dimension $d$.
Let $D$ and $A$ be big and nef $\bQ$-Cartier $\mathbb{Q}$-divisors on $X$.
Assume that $D^d>(2d)^d$. Then there is a bounded family $\mathcal{P}$ of subvarieties of $X$ such that for each pair $x$, $y\in X$ of general closed points, there is a member $G$ of $\mathcal{P}$ and an effective divisor $\Delta\sim_{\mathbb{Q}}D+(d-1)A$ such that 
\begin{enumerate}
\item $(X,B+\Delta)$ is lc near $x$ with a unique non-klt place whose centre is $G$,
\item $(X,B+\Delta)$ is not klt at $y$,
\item either $\dim G=0$ or $A^{d-\dim G}\cdot G\leq d^d$.
\end{enumerate}
\end{prop}
\begin{proof}

First, by {\cite[7.1]{HMX14}}, there is a bounded family $\mathcal{P}_0$ of subvarieties of $X$ such that for each pair $x$, $y\in X$ of general closed points, there is a member $G_0$ of $\mathcal{P}_0$ and an effective divisor $\Delta_0\sim_{\mathbb{Q}}D$ such that $(X,B+\Delta_0)$ is lc near $x$ with a unique non-klt place whose centre is $G_0$ and $(X,B+\Delta_0)$ is not klt at $y$.

Now suppose for some $0\leq i\leq d-2$, we are given a family $X$ $\mathcal{P}_i$ of subvarieties of $X$ such that for each pair $x$, $y\in X$ of general closed points, there is a member $G_i$ of $\mathcal{P}_i$ and an effective divisor $\Delta_i\sim_{\mathbb{Q}}D+iA$ such that 

\begin{enumerate}
\item $(X,B+\Delta_i)$ is lc near $x$ with a unique non-klt place whose centre is $G_i$,
\item $(X,B+\Delta_i)$ is not klt at $y$,
\item either $\dim G_i\leq d-i-1$ or $A^{d-\dim G_i}\cdot G_i\leq d^d$.
\end{enumerate}

If either $\dim G_i=0$ or $A^{d-\dim G_i}\cdot G_i\leq d^d$, then we set $G_{i+1}=G_i$ and $\Delta_{i+1}=\Delta_i+A$. 

On the other hand, if $\dim G_i>0$ and $A^{d-\dim G_i}\cdot G_i> d^d$, then we can write $A=A_i+E_i$ for some ample $\bQ$-Cartier $\bQ$-divisor $A_i$ and effective $\bQ$-divisor $E_i$ such that $A_i^{d-\dim G_i}\cdot G_i> d^d$. Note that such a decomposition can be done independently of $G_i$ because $G_i\in\mathcal{P}_i$ with $A_i^{d-\dim G_i}\cdot G_i> d^d$ form a bounded family. By {\cite[6.8.1 and 6.8.1.3]{Kol97}} and {\cite[2.32]{Bir16a}}, there are positive rational numbers $\delta\ll 1$ and $c<1$, such that there is an effective $\bQ$-divisor $H\sim_{\bQ}A_i$ such that
\begin{enumerate}
\item $(X,B+(1-\delta)\Delta_i+cH)$ is lc near $x$ with a unique non-klt place whose centre is $G'_{i}$,
\item $(X,B+(1-\delta)\Delta_i+cH)$ is not klt at $y$,
\item $\dim G_{i}'<\dim G_i$.
\end{enumerate}
We set ${G_{i+1}}={G'_{i}}$ and $\Delta_{i+1}=(1-\delta)\Delta_i+cH+\delta(D+iA)+(1-c)A_i+E_i\sim\Delta+A$ in this case. Note that $D$, $A$, $A_i$ and $E_i$ are independent of $x$ and $y$. 

Set $\mathcal{P}_{i+1}=\{G_{i+1}\}$, then the proposition follows from induction on $i$. Note that either $\dim G_i\leq d-i-1$ or $A^{d-\dim G_i}\cdot G_i\leq d^d$ implies $\mathcal{P}=\mathcal{P}_{d-1}$ is bounded.
\end{proof}

Next, we recall the following lemma by Jiang, which aims to cut down the dimension of $G$ to $0$ in the previous proposition. Jiang shows it in {\cite[3.1]{Jia17}} for $H=-K_X$ being ample. In fact, the proof works under the following setting.
\begin{lem}\label{jiang1}
Fix positive integers $d>k$. Let $(X,B)$ be a projective normal klt pair of dimension $d$ and $H$ be a nef and big $\bQ$-Cartier divisor on $X$. Assume there is a morphism $f:Y\rightarrow T$ of projective varieties with a surjective morphism $\phi :Y\rightarrow X$ such that a general fiber $F$ of $f$ is of dimension $k$ and $\phi|_F : F\rightarrow \phi(F)=G$ is birational, then
\[H^k\cdot G\geq\frac{\mathrm{lct}((X,B),|H|_{\bQ})^{d-k}}{\binom{d}{k}(d-k)^{d-k}}H^d.\]
\end{lem}
\begin{proof}
We follow the proof of {\cite[3.1]{Jia17}}. Taking normalizations and resolutions of $Y$ and $T$, we may assume they are smooth. Cutting by general hyperplane sections of $T$, we may assume $\phi$ is generically finite. Therefore, it holds that $\dim Y=d$.
Let $\mathcal{I}^{<m>}_G$ (resp. $\mathcal{I}^{<m>}_F$) be the sheaf of ideal of regular functions vanishing along $G$ (resp. $F$) to order at least $m$. Then  $\mathcal{I}^{<m>}_F=\mathcal{I}^{m}_F$, where $\mathcal{I}_F$ denotes the ideal sheaf of $F$.
So we have a natural injection
\[\mathcal{O}_X/\mathcal{I}^{<m>}_G\rightarrow\phi_*(\mathcal{O}_Y/\mathcal{I}^{m}_F).\]

Now we consider a rational number $l>0$ and a positive integer $m$ such that $lmH$ is Cartier. By the projection formula, we have
\begin{align*}
&h^0(X,\mathcal{O}_X(lmH)\otimes\mathcal{O}_X/\mathcal{I}^{<m>}_G)\\
\geq&h^0(X,\mathcal{O}_X(lmH)\otimes\phi_*(\mathcal{O}_Y/\mathcal{I}^{m}_F))\\
=&h^0(Y,\phi^*\mathcal{O}_X(lmH)\otimes\mathcal{O}_Y/\mathcal{I}^{m}_F).
\end{align*}
On the other hand, since $F$ is a general fiber of $f$, the conormal sheaf of $F$ is trivial. That is, we have $\mathcal{I}/\mathcal{I}^2\simeq\mathcal{O}_F^{\oplus(d-k)}$. Furthermore, we have
\[\mathcal{I}^{i-1}/\mathcal{I}^i\simeq S^{i-1}(\mathcal{I}/\mathcal{I}^2)\simeq\mathcal{O}_F^{\oplus\binom{d-k+i-2}{d-k-1}}\]
for every $i\geq 1$ (see {\cite [II. Theorem 8.24]{hart}}). Hence,
\begin{align*}
&h^0(Y,\phi^*\mathcal{O}_X(lmH)\otimes\mathcal{O}_Y/\mathcal{I}^{m}_F)\\
\leq&\sum_{i=1}^mh^0(Y,\phi^*\mathcal{O}_X(lmH)\otimes\mathcal{I}^{i-1}_F)/\mathcal{I}^{i}_F)\\
=&\sum_{i=1}^m\binom{d-k+i-2}{d-k-1}h^0(Y,\phi^*\mathcal{O}_X(lmH)\otimes\mathcal{O}_F)\\
=&\binom{d-k+m-1}{d-k}h^0(F,\phi^*\mathcal{O}_X(lmH)|_F).\\
\end{align*}
Now we consider the exact sequence
\[0\rightarrow\mathcal{O}_X(lmH)\otimes\mathcal{I}^{<m>}_G\rightarrow\mathcal{O}_X(lmH)\rightarrow\mathcal{O}_X(lmH)\otimes\mathcal{O}_X/\mathcal{I}^{<m>}_G\rightarrow0,\]
which implies
\begin{align*}
&h^0(X,\mathcal{O}_X(lmH)\otimes\mathcal{I}^{<m>}_G)\\
\geq&h^0(X,\mathcal{O}_X(lmH))-h^0(X,\mathcal{O}_X(lmH)\otimes\mathcal{O}_X/\mathcal{I}^{<m>}_G)\\
\geq&h^0(X,\mathcal{O}_X(lmH))-\binom{d-k+m-1}{d-k}h^0(F,\phi^*\mathcal{O}_X(lmH)|_F).\\
\end{align*}
We fix $l$ and consider the asymptotic behavior of the last two terms as $m$ goes to infinity.
By the definition of volume,
\[\lim_{m\to\infty}\frac{d!}{m^d}h^0(X,\mathcal{O}_X(lmH))=\mathrm{vol}(\mathcal{O}_X(lH))=l^dH^d.\]
On the other hand,
\begin{align*}
&\lim_{m\to\infty}\frac{d!}{m^d}\binom{d-k+m-1}{d-k}h^0(F,\phi^*\mathcal{O}_X(lmH)|_F)\\
=&\lim_{m\to\infty}\frac{d!m^{d-k}}{m^d(d-k)!}\frac{\mathrm{vol}(\phi^*\mathcal{O}_X(lH)|_F)m^k}{k!}\\
=&\binom{d}{k}\mathrm{vol}(\phi^*\mathcal{O}_X(lH)|_F)\\
=&\binom{d}{k}(\phi^*(lH)|_F)^k\\
=&\binom{d}{k}(\phi^*(lH))^k\cdot F\\
=&\binom{d}{k}(lH)^k\cdot G\\
=&\binom{d}{k}l^k(H)^k\cdot G.
\end{align*}
Consequently, for $l>\sqrt[d-k]{\binom{d}{k}\frac{(H)^k\cdot G}{H^d}}$ and $m$ sufficiently large, we have
\[h^0(X,\mathcal{O}_X(lmH))>\binom{d-k+m-1}{d-k}h^0(F,\phi^*\mathcal{O}_X(lmH)|_F).\]
Therefore $h^0(X,\mathcal{O}_X(lmH)\otimes\mathcal{I}^{<m>}_G)>0$, so there is an effective $\bQ$-divisor $D\sim_{\bQ}H$ such that $\mathrm{mult}_GD\geq1.$
Let $X_{sm}$ be the smooth locus of $X$. Since $\phi$ is surjective and $G$ is the image of a general fiber $F$ of $f$, $G|_{X_{sm}}$ is not empty. By {\cite [Lemma 2.29] {KM}}, $({X_{sm}},(d-k)D|_{X_{sm}})$ is not klt along  $G|_{X_{sm}}$. So $(X,B+(d-k)D)$ is not klt. Hence,
\[(d-k)l\geq\mathrm{lct}((X,B),\frac{1}{l}D)\geq\mathrm{lct}((X,B),|H|_{\bQ}).\]
Since $l$ is chosen arbitrarily such that $l>\sqrt[d-k]{\binom{d}{k}\frac{(H)^k\cdot G}{H^d}}$, we have
\[(d-k)\sqrt[d-k]{\binom{d}{k}\frac{(H)^k\cdot G}{H^d}}\geq\mathrm{lct}((X,B),|H|_{\bQ}).\]
That is,
\[H^k\cdot G\geq\frac{\mathrm{lct}((X,B),|H|_{\bQ})^{d-k}}{\binom{d}{k}(d-k)^{d-k}}H^d\]
holds.
\end{proof}

\begin{lem}\label{jiang}
Under the setting of Lemma \ref{jiang1}, there is a proper closed subset $S$ of $T$, such that
\[H^k\cdot \phi(F')\geq\frac{\mathrm{lct}((X,B),|H|_{\bQ})^{d-k}}{\binom{d}{k}(d-k)^{d-k}}H^d\]
for every fiber $F'$ of $f$ over $T-S$, and $\phi|_{f^{-1}(S)}$ does not dominate $X$.
\end{lem}
\begin{proof}
We apply the Noetherian induction (see, for example, {\cite[II, exercise 3.16]{hart}}) on Lemma \ref{jiang1}. Assume that $T'$ is a closed subset of $T$. Assume that for any proper closed subset $U$ of $T'$ there is a closed subset $V$of $U$ such that 
\[H^k\cdot \phi(F')\geq\frac{\mathrm{lct}((X,B),|H|_{\bQ})^{d-k}}{\binom{d}{k}(d-k)^{d-k}}H^d\]
for every fiber $F'$ of $f$ over $U-V$, and $\phi|_{f^{-1}(V)}$ does not dominate $X$. If $\phi|_{f^{-1}(T')}$ does not dominate $X$, we set $S'=T'$.
On the other hand, if $\phi|_{f^{-1}(T')}$ dominates $X$, by Lemma \ref{jiang1}, there is a proper closed subset $U'$ of $T'$, such that
\[H^k\cdot \phi(F')\geq\frac{\mathrm{lct}((X,B),|H|_{\bQ})^{d-k}}{\binom{d}{k}(d-k)^{d-k}}H^d\]
for every fiber $F'$ of $f$ over $T'-U'$. By the assumption, there is a closed subset $V'$of $U'$ such that 
\[H^k\cdot \phi(F')\geq\frac{\mathrm{lct}((X,B),|H|_{\bQ})^{d-k}}{\binom{d}{k}(d-k)^{d-k}}H^d\]
for every fiber $F'$ of $f$ over $U'-V'$, and $\phi|_{f^{-1}(V')}$ does not dominate $X$. We set $S'=V'$ in this case.

Now we have a closed subset $S'$ of $T'$, such that 
\[H^k\cdot \phi(F')\geq\frac{\mathrm{lct}((X,B),|H|_{\bQ})^{d-k}}{\binom{d}{k}(d-k)^{d-k}}H^d\]
for every fiber $F'$ of $f$ over $T'-S'$, and $\phi|_{f^{-1}(S')}$ does not dominate $X$, and we are done. Note that $S$ must be a proper subset because $\phi$ is surjective.

\end{proof}

%% file: 3.pf_of_main_main2.tex
\section{Proofs of Theorems}\label{sec 3}
Now we restate and prove the theorems in Section 1.
\begin{thm}\label{thm2}
Fix a positive integer $d$ and a positive real number $\theta$. Then there is a number $m$ depending only on $d$ and $\theta$ such that if $X$ is a projective normal variety satisfying
\begin{enumerate}
\item $\dim X=d,$
\item there exists a boundary $B$ such that $(X, B)$ is klt,
\item there is a nef $\bQ$-Cartier $\bQ$-divisor $H$ on $X$ with $\mathrm{lct}((X,B),|H|_{\bQ})>\theta$,
\item $H^d>\theta$,
\end{enumerate}
then $|K_X+\ceil{B+mH}|$ defines a birational map. Moreover, if $X$ is $\bQ$-factorial, then $|K_X+\ceil{mH}|$ defines a birational map.
\end{thm}
\begin{proof}
Suppose we are given $X$, $B$ and $H$ as in the assumption.

Set 
\begin{align*}
& q=\ceil{\frac{2d}{\sqrt[d]{\theta}}},\\
& p=\max_{1 \leq k \leq d-1}\left \{ \ceil{\sqrt[k]{\frac{\binom{d}{k}(d-k)^{d-k}d^d}{\theta^{d-k+1}}}}\right \}.
\end{align*}
Then by construction, $(qH)^d>(2d)^d$. 

Now apply Proposition \ref{birkar} with $D=qH$ and $A=pH$ and we have a bounded family $\mathcal{P}$ of subvarieties of $X$ such that for each pair $x,y\in X$ of general closed points, there is a member $G_{xy}$ of $\mathcal{P}$ and an effective divisor $\Delta\sim_{\mathbb{Q}}D+(d-1)A$ such that 
\begin{enumerate}
\item $(X,B+\Delta)$ is lc near $x$ with a unique non-klt place whose centre is $G_{xy}$,
\item $(X,B+\Delta)$ is not klt at $y$,
\item either $\dim G_{xy}=0$ or $A^{d-\dim G_{xy}}\cdot G_{xy}\leq d^d$.
\end{enumerate}

By {\cite [Lemma 2.21]{Bir16a}}, this means that there is a finite set $\{\phi_j :V_j\rightarrow T_j\}$ of projective varieties with surjective morphisms $\pi_j:V_j\rightarrow X$ such that each member $G\in\mathcal{P}$ is isomorphis through $\pi_j$ to a fiber of $\phi_j$ for some $j$.

Now consider a general fiber $G'$ of $\phi_j$ for each $j$. If $\dim G'>0$. By Lemma \ref{jiang}, there is a closed subset $S_j$ of $T_j$, such that
\[H^k\cdot G\geq\frac{\mathrm{lct}((X,B),|H|_{\bQ})^{d-k}}{\binom{d}{k}(d-k)^{d-k}}H^d>\frac{\theta^{d-k+1}}{\binom{d}{k}(d-k)^{d-k}}
\]
for every image $G$ of $\pi_j$ of a fiber $F'$ of $\phi_j$ over $T_j-S_j$, and $\pi_j|_{\phi_j^{-1}(S_j)}$ is not dominant. On the other hand, we set $S_j$ to be the empty set if $\dim G'=0$.

Now since $x$ and $y$ are general, they are not in $\pi_j(\phi_j^{-1}(S_j))$ for each $j$. So $G_{xy}$ is image of $\pi_j$ of a fiber of $\phi_j$ over $T_j-S_j$ for some $j$.
Suppose $\dim G_{xy}>0$, then by the definition of $p$, $A^k\cdot G'=(pH)^k\cdot G'>d^d.$ This contradicts (3) of Proposition \ref{birkar} and thus $\dim G'=0$ for any general fiber $G'$ of $\phi_j$ for each $j$.

Now $D+(d-1)A=(q+(d-1)p)H$. So $B+(q+(d-1)p+1)H$ is potentially birational and hence $|K_X+\ceil{B+(q+(d-1)p+1)H}|$ defines a birational map. Let $m=q+(d-1)p+1$ and we are done with the first statement.

If moreover, $X$ is $\bQ$-factorial, then $(X,0)$ is klt and $\mathrm{lct}((X,0),|H|_{\bQ})\geq\mathrm{lct}((X,B),|H|_{\bQ})>\theta$. Replacing $B$ by $0$, we are done.
\end{proof}

\begin{thm}\label{thm1}
Fix a positive integer $d$. There exists a number $M$, depending only on $d$, such that for any projective normal pair $(X,B)$ and for any big and nef $\bQ$-Cartier $\bQ$-divisor $H$ on $X$ satisfying
\begin{enumerate}
\item $\dim X=d,$
\item $(X, B)$ is klt,
\end{enumerate}
we have 
\[\mathrm{lct}((X,B),|H|_{\bQ})^{d-1}\overline{\tau}((X,B),H)H^d\leq M.\]
\end{thm}
\begin{proof}
By the linearity of lct, $\overline{\tau}$ and the volume function with respect to $H$, we may assume that $\overline{\tau}((X,B),H)=2$. So that $-(K_X+B)-H=-(K_X+B)-2H+H$ is big and then we can write $-(K_X+B)-H=F+E$ for some ample $\bQ$-Cartier $\bQ$-divisor $F$ and effective $\bQ$-divisor $E$. Assume the theorem fail to hold. Then for any number $M$, there is a pair $(X,B)$ a big and nef $\bQ$-Cartier $H$ on $X$ satisfying the assumptions, such that \[\mathrm{lct}((X,B),|H|_{\bQ})^{d-1}H^d=\mathrm{lct}((X,B),|H|_{\bQ})^{d-1}\frac{\overline{\tau}((X,B),H)}{2}H^d> \frac{M}{2}.\]
By Lemma \ref{kollar}, we have
\[\mathrm{lct}((X,B),|H|_{\bQ})^{d-1}H^d\leq \frac{d^d}{\mathrm{lct}((X,B),|H|_{\bQ})}.\] So we may assume $\text{lct}((X,B),|H|)\leq 1$ and $H^d>\frac{M}{2}\geq (4d)^d$.
Moreover, for every $0<k<d$, we may assume \[M>2\binom{d}{k}(d-k)^{d-k}(2(d-1)d)^d.\]

Now apply Proposition \ref{birkar} with $D=\frac{1}{2}H$ and $A=\frac{1}{2(d-1)}H$.  We have a bounded family $\mathcal{P}$ of subvarieties of $X$ such that for each pair $x,y\in X$ of general closed points, there is a member $G_{xy}$ of $\mathcal{P}$ and an effective divisor $\Delta\sim_{\mathbb{Q}}D+(d-1)A=H$ such that 
\begin{enumerate}
\item $(X,B+\Delta)$ is lc near $x$ with a unique non-klt place whose centre is $G_{xy}$,
\item $(X,B+\Delta)$ is not klt at $y$,
\item either $\dim G_{xy}=0$ or $A^{d-\dim G_{xy}}\cdot G_{xy}\leq d^d$.
\end{enumerate}

By {\cite [Lemma 2.21]{Bir16a}}, this means that there is a finite set $\{\phi_j :V_j\rightarrow T_j\}$ of projective varieties with surjective morphisms $\pi_j:V_j\rightarrow X$ such that each member $G\in\mathcal{P}$ is isomorphis through $\pi_j$ to a fiber of $\phi_j$ for some $j$.

Now consider a general fiber $G'$ of $\phi_j$ for each $j$. If $\dim G'>0$. By Lemma \ref{jiang}, there is a closed subset $S_j$ of $T_j$, such that
\[H^k\cdot G\geq\frac{\mathrm{lct}((X,B),|H|_{\bQ})^{d-k}}{\binom{d}{k}(d-k)^{d-k}}H^d
\]
for every image $G$ of $\pi_j$ of a fiber $F'$ of $\phi_j$ over $T_j-S_j$, and $\pi_j|_{\phi_j^{-1}(S_j)}$ is not dominant. On the other hand, we set $S_j$ to be the empty set if $\dim G'=0$.

Now since $x$ and $y$ are general, they are not in $\pi_j(\phi_j^{-1}(S_j))$ for each $j$. So $G_{xy}$ is image of $\pi_j$ of a fiber of $\phi_j$ over $T_j-S_j$ for some $j$.
Suppose $\dim G_{xy}>0$, then we have
 \begin{align*}
H^k\cdot G_{xy}&\geq\frac{\mathrm{lct}((X,B),|H|_{\bQ})^{d-k}}{\binom{d}{k}(d-k)^{d-k}}H^d\\
&\geq\frac{\mathrm{lct}((X,B),|H|_{\bQ})^{d-1}}{\binom{d}{k}(d-k)^{d-k}}H^d\\
&>\frac{M}{2\binom{d}{k}(d-k)^{d-k}}>(2(d-1)d)^d.
\end{align*}
By construction, $A^k\cdot G_{xy}=(\frac{1}{2(d-1)})^kH^k\cdot G_{xy}>d^d.$ This contradicts (3) of the above and thus $\dim G_{xy}=0$.

Now for general $x$ and $y$, Nklt$(X,B+\Delta)$ contains $y$ and $x$ and isolates $x$. Since $x$ is general, Nklt$(X,B+\Delta+E)$ also isolates $x$. By Lemma \ref{cnnlem}, $-K_X-(B+\Delta+E)\sim_{\mathbb{Q}}-(K_X+B)-H-E=F$ is not ample, which is a contradiction.
\end{proof}

We recall the following theorem by Hacon and Xu.
\begin{thm}$\mathrm{(}${\cite [Theorem 1.3]{HX15}}$\mathrm{)}$\label{hx}
Fix a positive integer $d$ and a DCC set $\mathscr{I}$ of rational numbers in $[0,1]$. The set of all projective normal pairs $(X,B)$ satisfying
\begin{enumerate}
\item $(X,B)$ is klt log Calabi--Yau of dimension $d$,
\item $B$ is big,
\item the coefficients of $B\in\mathscr{I}$,
\end{enumerate}
forms a bounded family.
\end{thm}
Then we show the following log-version of {\cite [Lemma 2.26]{Bir16a}}.
\begin{lem}\label{qfact}
Fix positive integers $d$ and $k$. Let  $\mathcal{P}$ be a set of klt weak Fano pairs of dimension $d$. Assume that for every element $(Y,B_Y)\in\mathcal{P}$, there is a $k$-complement $K_Y+B^+_Y$ of $K_Y+B_Y$ such that $(X,B^+_Y)$ is klt and $B^+_Y\geq B_Y$. Let  $\mathcal{Q}$ be the set of normal projective pairs $(X,B)$ such that
\begin{enumerate}
\item there is $(Y,B_Y)\in\mathcal{P}$ and a birational map $X\dashrightarrow Y$,
\item there is a common resolution $\phi :W\rightarrow Y$ and $\psi:W\rightarrow X$,
\item $\phi^*(K_Y+B_Y)\geq\psi^*(K_X+B)$,
\end{enumerate}
Then for every element $(X,B)\in\mathcal{Q}$, there is a $k$-complement $K_X+B^+$ of $K_X+B$ such that $(X,B^+)$ is klt and $B^+\geq B$.
\end{lem}
\begin{proof}
Let $K_X+B^+$ be the crepant pullback of $K_Y+B_Y^+$ to $X$. Then $(X,B^+)$ is klt. Since $\phi^*(K_Y+B_Y)\geq\psi^*(K_X+B)$, $B^+-B\geq\psi_*\phi^*(B_Y^+-B_Y)\geq 0$ is big and $K_X+B^+$  is an $k$-complement of $K_X+B$.

\end{proof}

Next, we recall the following proposition of Birkar with a small observation.
\begin{prop}\label{birkar2}$\mathrm{(}${\cite[Proposition 4.4]{Bir16a}}$\mathrm{)}$
Fix positive integers $d$, $v$ and a positive real number $\epsilon$. Then there exists a bounded set of couples $\mathcal{P}$ and a positive real number $c$ depending only on $d$, $v$ and $\epsilon$ satisfies the following. Assume 
\begin{itemize}
	\item $X$ is a normal projective variety of dimension $d$,
	\item $B$ is an effective $\bR$-divisor with coefficient at least $\epsilon$,
	\item $M$ is a $\bQ$-divisor with $|M|$ defining a birational map,
	\item $M-(K_X+B)$ is pseudo-effective,
	\item $\mathrm{vol}(M)<v$, and
	\item $\mu_D(B+M)\geq 1$ for every component $D$ of $M$.
\end{itemize}
Then there is a projective log smooth couple $(\overline{W}, \Sigma_{\overline{W}})\in \mathcal{P}$, a birational map $\overline{W}\dashrightarrow X$ and a common resolution $X'$  of this map such that
\begin{enumerate}
\item Supp$\Sigma_{\overline{W}}$ contains the exceptional divisor of $\overline{W}\dashrightarrow X$ and the birational transform of Supp$(B+M)$,
\item there is a resolution $\phi : W \rightarrow X$ such that $M_W:=M|_W\sim A_W+R_W$ where $A_W$ is the movable part of $|M_W|$, $|A_W|$ is base point free, $X'\rightarrow X$ factors through $W$ and $A_{X'}:=A_W|_{X'}\sim 0/\overline{W}$.
\end{enumerate}
Moreover, if M is nef and $M_{\overline{W}}$ is the pushdown of $M_{X'}:=M|_{X'}$. Then each coefficient of $M_{\overline{W}}$ is at most $c$.
\end{prop}
Note that in the original statement of {\cite[Proposition 4.4]{Bir16a}}, $M$ is assumed to be nef. We observe from Birkar's proof that the nefness of $M$ is used only when showing the existence of $c$ and is not necessary when showing (1) and (2) of proposition \ref{birkar2}.

Now we are ready to show the main theorem of this paper.
The idea is to follow the strategy of {\cite[Proposition 7.13]{Bir16a}}, which is to construct a klt complement with coefficients in a finite set depending only on $d$, $\theta$ and $\mathscr{R}$, and then apply Theorem \ref{hx}.

\begin{thm}\label{thm4.2}
Fix a positive integer $d$, a positive real number $\theta$ and a finite set $\mathscr{R}$ of rational numbers in $[0,1]$. The set $\mathcal{D}$ of all klt weak Fano pairs $(X,B)$ satisfying
\begin{enumerate}
\item $\dim X=d$,
\item the coefficients of $B\in\Phi(\mathscr{R})$,
\item $\alpha(X,B)>\theta$,
\item $(-(K_X+B))^d>\theta$,
\end{enumerate}
forms a log bounded family.
Moreover, there is a finite set of rational numbers $I\subseteq [0,1]$ depending only on $d$, $\theta$ and $\mathscr{R}$ such that for every element $(X,B)\in\mathcal{D}$, there is a $\mathbb{Q}$-divisor $\Theta\geq B$ with a klt log Calabi--Yau pair $(X,\Theta)$.
\end{thm}

\begin{proof}
By Theorem \ref{hx}, it is enough to show the existence of $I$.

By Lemma \ref{qfact}, replacing by a small $\bQ$-factorialisation of $X$, we may assume $X$ is $\bQ$-factorial.

By Theorem \ref{bddcomp}, there is a positive integer $n$ depending only on $d$ and $\mathscr{R}$, such that there is an $n$-complement $K_X+B^+$ of $K_X+B$ such  that $B^+\geq B$. If $(X,B^+)$ is klt, then the theorem follows by Theorem \ref{hx}. So we may assume $(X,B^+)$ is not klt.

Let $m$ be given by Theorem \ref{thm2} such that  $|K_X+\ceil{m(-K_X-B)}|$ defines a birational map. Replacing $n$ and $m$ by $2mn$, we may assume $n=m>1$. Since $1$ is the only possible accumulating point of $\Phi(\mathscr{R})$, there are only finitely many elements in $\Phi(\mathscr{R})\cap[0,\frac{m-1}{m}]$. Let $I$ be a positive integer such that $\Phi(\mathscr{R})\cap[0,\frac{m-1}{m}]\subseteq \frac{1}{I}\mathbb{Z}$. Replacing $m$ by $Im$, we may assume that $I$ divides $m$ and $\mathrm{Supp}(\{m(B^+-B)\})\subseteq\mathrm{Supp}\floor{B^+}\cap$Supp$B$. Note that now $K_X+B^+$ is also an $\frac{m}{I}$-complement of  $K_X+B$.

Note that this also implies $-B^++\ceil{m(B^+-B)}\leq\floor{m(B^+-B)}$. Since $K_X+B^+$ is an $m$-complement, we have $mK_X\sim mB^+$. Hence, $|m\floor{m(B^+-B)}|$ defines a birational map. Since $m\floor{m(B^+-B)}\leq \floor{m^2(B^+-B)}$, replacing $m$ by $m^2$, we may assume $|\floor{m(B^+-B)}|$ defines a birational map. On the other hand, by Lemma \ref{kollar}, vol$(\floor{m(B^+-B)})<v$ for some $v$ depending only on $d$, $m$ and $\theta$.

Let $M$ be a general element of $|\floor{m(B^+-B)}|$. By Proposition \ref{birkar2}, there is a bounded set of couples $\mathcal{P}$ depending only on $d$, $m$ and $\theta$, such that there is a projective log smooth couple $(\overline{W}, \Sigma_{\overline{W}})\in \mathcal{P}$, a birational map $\overline{W}\dashrightarrow X$ and a common resolution $X'$  of this map such that
\begin{enumerate}
\item Supp$\Sigma_{\overline{W}}$ contains the exceptional divisor of $\overline{W}\dashrightarrow X$ and the birational transform of Supp$(B^++M)$,
\item there is a resolution $\phi : W \rightarrow X$ such that $M_W:=M|_W\sim A_W+R_W$ where $A_W$ is the movable part of $|M_W|$, $|A_W|$ is base point free, $X'\rightarrow X$ factors through $W$ and $A_{X'}:=A_W|_{X'}\sim 0/\overline{W}$.
\end{enumerate}
Since $M$ is a general element of $|\floor{m(B^+-B)}|$, we may assume $M_{W}=A_{W}+R_W$ and $A_{W}$ is general in $|A_{W}|$. 
In particular, if $A_{\overline{W}}$ is the pushdown of $A_{W}|_{X'}$ to $\overline{W}$, then $A_{\overline{W}}\leq\Sigma_{\overline{W}}$.
Let $M$, $A$, $R$ be the pushdowns of $M_W$, $A_W$, $R_W$ to $X$.



Since $|A_{\overline{W}}|$ defines a birational contraction and $A_{\overline{W}}\leq\Sigma_{\overline{W}}$, there exists $l\in \mathbb{N}$ depending only on $\mathcal{P}$ such that $lA_{\overline{W}}\sim G_{\overline{W}}$ for some $G_{\overline{W}}\geq 0$ whose support contains $\Sigma_{\overline{W}}$. Let $K_{\overline{W}}+B^+_{\overline{W}}$ be the crepant pullback of $K_X+B^+$ to ${\overline{W}}$. Then $(\overline{W},B^+_{\overline{W}})$ is sub-lc and
$$\text{Supp}B^+_{\overline{W}}\subseteq\text{Supp}\Sigma_{\overline{W}}\subseteq \text{Supp}G_{\overline{W}}.$$
Let $G$ be the pushdown of $G_{X'}:=G_{\overline{W}}|_{X'}$ to $X$. 
Since $A_{X'}$ is the pullback of $A_{\overline{W}}$, $lA_{X'}\sim G_{X'}$ and $lA\sim G$. Therefore, $G+lR+l\{m(B^+-B)\}\sim_{\bQ}m(B^+-B)$.

Take a positive rational number $t\leq (lm)^{-d}\theta$, then \[(X,B+t(G+lR+l\{m(B^+-B)\}))\] is klt. Moreover, we have
\[-K_X-B-t(G+lR+l\{m(B^+-B)\})\sim_{\bQ}B^+-B-t(lm(B^+-B)).\]
By replacing $t$, we may assume $t<\frac{1}{lm}$.
Since 
\[B^+-B-t(lm(B^+-B))=(1-tlm)(B^+-B)\geq 0,\]
$B^+-B-t(lm(B^+-B))$ is nef and big. Therefore \[(X,B+t(G+lR+l\{m(B^+-B)\}))\] is klt weak Fano.

Now we argue that the coefficients of $B+t(G+lR+l\{m(B^+-B)\})$ are in a set of hyperstandard multiplicities associated to a finite set.
Write \[B+tl\{m(B^+-B)\}=(1-tlm)B+tlmB^+-tl\floor{m(B^+-B)}.\]
Denote $B^{\leq 1-\frac{1}{m}}$ to be the sum of irreducible components of $B$ with coefficients at most $1-\frac{1}{m}$ and $B^{> 1-\frac{1}{m}}=B-B^{\leq 1-\frac{1}{m}}$.
Since $1$ is the only possible accumulating point of $\Phi(\mathscr{R})$, $\mathscr{S}:=\Phi(\mathscr{R})\cap[0,\frac{m-1}{m}]$ is a finite set.
Therefore, the coefficients of $B^{\leq 1-\frac{1}{m}}\in\mathscr{S}$.
On the other hand, if $D$ is a component of $B^{> 1-\frac{1}{m}}$, then by assumption, $\mu_D(B)=1-\frac{b}{m'}>1-\frac{1}{m}$ for some $b\in\mathscr{R}$ and $m'\in\mathbb{N}$.
Since $B^+$ is an $m$-complement,  $\mu_D(B^+)\in \frac{1}{m}\mathbb{N}$. So $\mu_D(B^+)=1$ and $\mu_D(\floor{m(B^+-B)})=0.$
Then we have
\[\mu_D((1-tlm)B^{> 1-\frac{1}{m}}+tlmB^+-tl\floor{m(B^+-B)})\]\[=(1-tlm)(1-\frac{b}{m'})+tlm=1-\frac{(1-tlm)b}{m'}.\]
Therefore, the coefficients of $(1-tlm)B^{> 1-\frac{1}{m}}+tlmB^+-tl\floor{m(B^+-B)}\in\Phi((1-tlm)\mathscr{R})\cup tl\mathbb{N}$ since $B^+$ and $\floor{m(B^+-B)}$ are integral.
Since $B+tl\{m(B^+-B)\}=(1-tlm)B+tlmB^+-tl\floor{m(B^+-B)}=(1-tlm)B^{\leq 1-\frac{1}{m}}+(1-tlm)B^{> 1-\frac{1}{m}}+tlmB^+-tl\floor{m(B^+-B)}$, the coefficients of $B+tl\{m(B^+-B)\}\in((1-tlm)\mathscr{S}\cup\{0\}+tl\mathbb{N})\cup \Phi((1-tlm)\mathscr{R})$.
Consequently, since $G$ and $R$ are integral ,the coefficients of $B+t(G+lR+l\{m(B^+-B)\})$ belongs to
\[((1-tlm)\mathscr{S}\cup\Phi((1-tlm)\mathscr{R})\cup\{0\})+t(\mathbb{N}\cup\{0\}).\]
Moreover, since $(X,B+t(G+lR+l\{m(B^+-B)\}))$ is klt, the coefficients of $B+t(G+lR+l\{m(B^+-B)\})$ is less than $1$.
Since $(1-tlm)\mathscr{S}$ and $(1-tlm)\mathscr{R}$ are both finite, the only possible accumulating point of $\mathscr{T}:=((1-tlm)\mathscr{S}\cup\Phi((1-tlm)\mathscr{R})\cup\{0\})$ is $1$.
Thus, $(\mathscr{T}+t(\mathbb{N}\cup\{0\}))\cap [0,1]=\mathscr{T}\cup((\mathscr{T}\cap[0,1-t])+t(\mathbb{N}\cup\{0\}))\cap[0,1]=\mathscr{T}\cup\mathscr{Q}$, where $\mathscr{Q}:=((\mathscr{T}\cap[0,1-t])+t(\mathbb{N}\cup\{0\}))\cap[0,1]$ is a finite set.
The coefficients of $B+t(G+lR+l\{m(B^+-B)\})$ belongs to $\mathscr{T}\cup\mathscr{Q}=(1-tlm)\mathscr{S}\cup\Phi((1-tlm)\mathscr{R})\cup\{0\}\cup\mathscr{Q}\subseteq \Phi((1-tlm)\mathscr{R}\cup \mathscr{P})$, where $\mathscr{P}:=\{1\}-((1-tlm)\mathscr{S}\cup\{0\}\cup\mathscr{Q})$ is a finite set.

By Theorem \ref{bddcomp}, there is a positive integer $n'$ depending only on $d$, $\mathscr{R}$ $m$, $l$ and $t$ such that there is an $n'$-complement $K_X+\Omega$ of $K_X+B+t(G+lR+l\{m(B^+-B)\})$, such that
\[\Omega\geq B+t(G+lR+l\{m(B^+-B)\}).\]

On the other hand, let
$$\Delta_{\overline{W}}:=B^+_{\overline{W}}+\frac{t}{m}A_{\overline{W}}-\frac{t}{lm}G_{\overline{W}}.$$
Then $(\overline{W},\Delta_{\overline{W}})$ is sub-$\epsilon$-klt for some $\epsilon>0$ depending only on $\mathcal{P}$, $t$, $l$ and $m$ since $\mathrm{Supp}B^+_{\overline{W}}\subseteq\mathrm{Supp}\Sigma_{\overline{W}}\subseteq\mathrm{Supp}G_{\overline{W}}$, $(\overline{W}, \Sigma_{\overline{W}})$ is log smooth, $(\overline{W},B^+_{\overline{W}})$ is sub-lc, and $A_{\overline{W}}$ is not a component of $\ceil{B^+_{\overline{W}}}$. Moreover, $K_{\overline{W}}+\Delta_{\overline{W}}\sim_{\bQ} 0$.

Let
$$\Delta:=B^++\frac{t}{m}A-\frac{t}{lm}G.$$
Then $K_X+\Delta\sim_{\bQ} 0$. $(X,\Delta)$ is sub-klt since $K_X+\Delta$ is the crepant pullback of $K_{\overline{W}}+\Delta_{\overline{W}}$.

Let $\Theta=\frac{1}{2}\Delta+\frac{1}{2}\Omega$. Then
$$\Theta=\frac{1}{2}B^++\frac{t}{2m}A-\frac{t}{2lm}G+\frac{1}{2}\Omega$$
$$\geq\frac{1}{2}B^++\frac{t}{2m}A-\frac{t}{2lm}G+\frac{1}{2}B+\frac{t}{2}(G+lR)\geq-\frac{t}{2lm}G+\frac{t}{2}G\geq 0.$$

Since $(X,\Delta)$ is sub-klt, $K_X+\Delta\sim_{\bQ} 0$ and $(X,\Omega)$ is lc log Calabi--Yau, $(X,\Theta)$ is klt log Calabi--Yau. The coefficients of $\Theta$ belong to a fixed finite set $I$ depending only on $t$, $l$, $m$ and $n'$. Moreover, Supp$B\subseteq$Supp$\Omega\subseteq$Supp$\Theta$. By Theorem \ref{hx}, we are done.
\end{proof}

\begin{thm}\label{main}
Fix a positive integer $d$, positive real numbers $\theta$ and $\delta$ and a finite set $\mathscr{R}$ of rational numbers in $[0,1]$. The set of all klt Fano pairs $(X,B)$ satisfying
\begin{enumerate}
\item $\dim X=d$,
\item the coefficients of $B\in\Phi(\mathscr{R})$,
\item $\alpha(X,B)^{d-1+\delta}(-(K_X+B))^d>\theta$,
\end{enumerate}
forms a log bounded family.
\end{thm}

\begin{proof}
This follows from Theorem \ref{excbdd}, Theorem \ref{thm1} and Theorem \ref{thm4.2}.
\end{proof}

\begin{cor}
Fix a positive integer $d$ and two positive real numbers $\delta$ and $\theta$. Then the set of klt-Fano varieties $X$ satisfying
\begin{enumerate}
\item $\dim X=d$,
\item $\alpha (X)^{d-1+\delta}(-K_X)^d>\theta$,
\end{enumerate}
forms a bounded family.
\end{cor}
\begin{proof}
This follows either from the previous theorem or from Theorem \ref{excbdd}, Theorem \ref{thm1} and Theorem \ref{jiangbdd}.
\end{proof}